\documentclass[10pt, reqno]{amsart}

\usepackage[utf8]{inputenc}
\usepackage{amsmath}
\usepackage{amsthm}
\usepackage{enumerate}
\usepackage{float}
\usepackage{amsfonts}
\usepackage{verbatim}
\usepackage{graphicx}
\usepackage[foot]{amsaddr}
\newcommand{\ind}{1\hspace{-2.5mm}{1}}
\usepackage[margin=3.5cm]{geometry}

\theoremstyle{plain}
\newtheorem{theorem}{Theorem}
\newtheorem{proposition}{Proposition}
\newtheorem{corollary}{Corollary}
\newtheorem{lemma}{Lemma}

\theoremstyle{remark}
\newtheorem*{remark}{Remark}

\theoremstyle{definition}
\newtheorem{definition}{Definition}

\begin{document}
	
	\author[Sebastian Rosengren and Pieter Trapman]{Sebastian Rosengren$^1$ and Pieter Trapman$^2$}      
	\address{$^{1,}$$^{2}$Department of Mathematics, Stockholm University, 106 91 Stockholm, Sweden.}
	\email{$^1$rosengren@math.su.se, $^2$ptrapman@math.su.se}

	\title[Dynamic Erd\H{o}s-Rényi Graph]{A Dynamic Erd\H{o}s-Rényi Graph Model} 
	\maketitle

	\begin{abstract}
		In this article we introduce a dynamic Erd\H{o}s-R\'enyi graph model, in which, independently for each vertex pair, edges appear and disappear according to a Markov on-off process.
		
		In studying the dynamic graph we present two results. The first being on how long it takes for the graph to reach stationarity. We give an explicit expression for this time, as well as proving that this is the fastest time to reach stationarity among all strong stationary times.
		
		The main result concerns the  time it takes  for the dynamic graph to reach a certain number of edges. We give an explicit expression for the expected value of such a time, as well as study its asymptotic behavior. This time is related to the first time the dynamic Erd\H{o}s-R\'enyi graph contains a cluster exceeding a certain size.
	\end{abstract}
	
	\keywords{
		Erd\H{o}s-Rényi Graph; 
		Markov Process; 
		fastest time to stationarity; 
		strong stationary times; 
		hitting times} 
	
	
	\section{Introduction}
	The Erd\H{o}s-R\'enyi graph, in this text called the \textit{static} Erd\H{o}s-R\'enyi graph, is a well-studied model for random graphs, which is either (i) consisting of $n$ vertices and $m$ edges, where the edges are assigned uniformly to the $\binom{n}{2}$ vertex pairs---this graph model is denoted $G(n,m)$; or (ii) consisting of $n$ vertices where edges are assigned independently between vertex pairs with probability $p$---this graph model is denoted $G(n,p)$, see \cite{Boll} for more details and many properties of the model.
	In this article we introduce a natural dynamic version of such a model: the dynamic Erd\H{o}s-R\'enyi graph.
	
	Before moving on we set some notation: Throughout $N=\binom{n}{2}$. Furthermore, we use the asymptotic order notation: $f(n) = O(g(n)) \iff  |f(n)|\leq M |g(n)|$ for large $n$ and some $M<\infty$; $f(n)=\Theta(g(n))$ if and only if  both $f(n)=O(g(n))$  and $g(n)=O(f(n))$. Finally, $f(n) = o(g(n)) \iff \lim\limits_{n\to \infty}\frac{|f(n)|}{|g(n)|}=0$. Furthermore we say that an event happens ``with high probability'' (w.h.p.) if the probability of the event converges to 1 as $n \to \infty$. Throughout (unless otherwise stated), all asymptotic's are for the limit $n \to \infty$.
	\subsection{The dynamic Erd\H{o}s-R\'enyi graph}
	For $\alpha, \beta>0$ and $n$ a positive integer, the dynamic Erd\H{o}s-R\'enyi graph $\{G(t),\ t\geq 0 \}$ is a stochastic process evolving according to the following dynamics,
	\begin{enumerate}[(i)]
		\item The number of vertices is fixed at $n$.
		\item Independently for each vertex pair, if no edge is present an edge is added after an Exp($\frac{\beta}{n-1}$)-distributed time; if an edge is present, the edge is removed after an Exp($\alpha$)-distributed time.
	\end{enumerate}
	Note that (ii) can be replaced by,
	\begin{enumerate}
		\item[(iia)] independently for each vertex pair, the state of an edge (present or not present) is updated at the points of a Poisson process with intensity $\lambda = \alpha+\frac{\beta}{n-1}$.  
		Independently of the Poisson process and previous states of the edges, with probability $p = \frac{\beta}{\beta+(n-1) \alpha}$ an edge will be present after the update, and with probability $q=1-p$ it will not be present.  
	\end{enumerate}
	
	\begin{remark}
		The choice of birth rate $\frac{\beta}{n-1}$ and not  $\frac{\beta}{n}$ is because if the birth rate equals $\frac{\beta}{n-1}$  then if $\beta = \alpha$, the dynamic graph converges to a \textit{critical} Erd\H{o}s-R\'enyi graph ($G(n,p)$ with $p = \frac{1}{n}$). However, for large $n$ it makes no difference which of the two birth rates is chosen. 
		
		We note that by changing the 
		time-unit we can (without loss of generality) set $\beta = 1$. We choose 
		to keep $\beta$, so that probabilities are always expressed as quotients 
		of rates or products of rates and times, while hitting times are always 
		expressed as the inverse of rates.
	\end{remark}

	\begin{remark}
		The dynamic Erd\H{o}s-R\'enyi graph model has been used in \cite{Altm95} (without using that name) as the underlying dynamic network structure on which an epidemic is spreading. In that paper the network is assumed to start in stationarity and apart from the stationary degree distribution no further properties of the dynamic graph process are studied.   
	\end{remark}

	Let $\{\chi_{u,v}(t),\ t\geq 0 \}$ denote the indicator process---being 1, when an edge is present between vertex $u$ and $v$ and 0 otherwise. We refer to this process as an edge process. This is, by definition, a birth-death process on $\{0,1\}$ with birth-rate $\lambda = \frac{\beta}{n-1}$ and death-rate $\mu = \alpha$, also known as an on-off process. We think of the dynamic Erd\H{o}s-R\'enyi graph $\{G(t),\ t\geq 0 \}$ as being composed of these i.i.d.\ processes, with $G(t) = (\chi_{1,2}(t),\dots, \chi_{n-1,n}(t))$.
	
	Throughout we assume that the underlying probability space has enough structure so that (iia) holds, i.e.\ that the edge processes are generated according to (iia). This means that the probability space has a filtration $\{\mathcal{F}_t,\ t\geq 0\}$, where $\mathcal{F}_t$ is the information generated by the update times and corresponding edge updates up to time $t$, and that $\{G(t),\ t\geq 0 \}$ is adapted to this filtration. We shall see that, for our purposes, this assumption can be made without loss of generality.
	\subsection{The fastest time to stationarity}
	\label{MRI}
	Below, we see that the distribution of $\{G(t),\ t\geq 0 \}$ converges to the distribution of a static Erd\H{o}s-R\'enyi graph with edge probability $p = \frac{\beta}{\beta + (n-1) \alpha}$, which is also the stationary distribution of the dynamic graph. In Section \ref{FTTS} we construct a time $T_s$, called the \textit{fastest time to stationarity} for $\{G(t),\ t\geq 0 \}$ \cite{Fill}, with the following properties,
	\begin{enumerate}[(i)]
		\item $G(T_s)$ is distributed according to the stationary distribution of $\{G(t),\ t\geq 0 \}$, and is independent of $T_s$;
		\item $\{G(T_s+s);\ s\geq 0 \}$ is a stationary process, and is independent of $T_s$;
		\item if $T'$ is any other random variable satisfying (i) and (ii) then for all $t>0$,
		\[
		\mathbb{P}(T<t) \geq \mathbb{P}(T'<t)
		\]
		i.e. $T$ is the stochastically smallest time satisfying (i) and (ii).
	\end{enumerate}
	
	A key part in constructing such a time to stationarity is noting that when an edge process $\{\chi_{u,v}(t),\ t\geq 0 \}$ enters stationarity it stays there. Since the dynamic graph, $G(t) = (\chi_{1,2}(t),\dots, \chi_{n-1,n}(t))$ is composed of edge processes and the processes are independent, the dynamic graph should  be in stationarity if all edge processes have entered stationarity. Hence we proceed by finding the fastest times to stationarity for the underlying edge processes, i.e. $\{T_{u,v},\ \forall (u,v) \}$, and then show that the maximum of these is indeed the fastest time to stationarity for the dynamic graph.
	
	In order to derive the stationary distribution for the dynamic graph process, we note that the following result for the underlying edge processes is immediate from defining property (iia) of the dynamic Erd\H{o}s-R\'enyi graph:
	
	\begin{lemma}
		\label{edgelemma}
		\begin{enumerate}
			\item[(a)] 	For $u,v \in V$, the edge processes $\{\chi_{u,v}(t),\ t\geq 0\}$ are independent ergodic Markov processes on $\{0,1\}$, with probability transition functions equal to,

			\begin{align*}
				&\mathbb{P}(\chi_{u,v}(t) = 1 |\chi_{u,v}(0)=0) = p_{0,1}(t) = \frac{\beta}{\beta + (n-1)\alpha}  \left(1-  e^{-(\alpha+\frac{\beta}{n-1}) t}\right) \\
				& \mathbb{P}(\chi_{u,v}(t) = 1 |\chi_{u,v}(0)=1) = p_{1,1}(t) =
				e^{-(\alpha+\frac{\beta}{n-1}) t} + \frac{\beta}{\beta + (n-1)\alpha}  \left(1-e^{-(\alpha+\frac{\beta}{n-1}) t}\right)
			\end{align*}
			and stationary distribution $\pi$ equal to,
			$$
			\pi(1) = \frac{\beta}{\beta + (n-1)\alpha} \qquad \text{and} \qquad
			\pi(0) = \frac{(n-1)\alpha}{\beta + (n-1)\alpha}.
			$$\\
			\item[(b)] 	The dynamic Erd\H{o}s-R\'enyi graph $\{G(t),\ t\geq 0 \}$ is an ergodic Markov process with finite state space and with unique stationary and limiting distribution equal to that of a static Erd\H{o}s-R\'enyi graph $G(n,p)$ with edge probability $p=\frac{\beta}{\beta + (n-1)\alpha}$.
		\end{enumerate}
		
	\end{lemma}

	By (iia) in the definition of $\{G(t),\ t\geq 0 \}$, we immediately see that after the first update, i.e.\ after an exponentially distributed time with parameter $\alpha+\frac{\beta}{n-1}$, an edge process is in stationarity. It will be proven in Section \ref{FTTS} that this is also the fastest time to stationarity for an edge process given that it starts in state $0$ or $1$. We deduce that the fastest time to stationarity $T_s$ for the dynamic graph is distributed as the maximum of $N=\binom{n}{2}$ (the number of edge processes) independent exponentially distributed random variables with parameter $\alpha+\frac{\beta}{n-1}$. The following will be proved in Section \ref{FTTS}:
	\begin{theorem}
		\label{Fastesttimetostationary} 
		Let $\{G(t),\ t\geq 0 \} $ be the dynamic Erd\H{o}s-R\'enyi graph starting in a given state, i.e. $P(G(0)=\sigma)=1$ for some $\sigma$ in the state space. Let $T_{u,v}$ be the fastest time to stationarity for $\{\chi_{u,v}(t),\ t\geq 0 \}$.
		Then,
		\[
		T_s = \max_{u,v} \{T_{u,v}\}
		\]
		is the fastest time to stationarity for the dynamic Erd\H{o}s-R\'enyi graph.
		Furthermore, its distribution function is given by,
		\begin{equation}
			\label{distfasttime}
			\mathbb{P}(T_s \leq  t) = \left( 1 - e^{-(\alpha+\frac{\beta}{n-1}) t}\right) ^{N}.
		\end{equation}
	\end{theorem}
	
	Concerning the asymptotic behavior of (\ref{distfasttime}), we show that it is very likely that the graph enters stationarity roughly at time $\frac{2 \log (n)}{\alpha}$.
	As $T_s$ is the maximum of $N$ independent exponentially distributed random variables, it follows that, properly scaled, it converges in distribution to a Gumbel distributed random variable. In order to be self-containing a proof for the following class-room result will be provided in Section \ref{FTTS}:
	\begin{corollary}
		\label{asymptoticlemma}
		Let $T_s$ be the fastest time to stationarity for the dynamic Erd\H{o}s-R\'enyi graph. Then for all $x\in \mathbb{R}$,
		$$
		\mathbb{P}( \alpha T_s - 2 \log (n) + \log 2 \leq  x) \to e^{-e^{-x}}  \text{ as } n \to \infty\\
		$$
		Furthermore, 
		$\mathbb{E}(T_s) = O(\log (n))$.
	\end{corollary}
	We note that for large $n$ and time $t>>\frac{2 \log (n)}{\alpha}$ we have $\mathbb{P}(T_s < t)\approx 1$, and in the time period $[0,t]$ the process is most of that time in stationarity. So in studying certain properties of the dynamic graph one may be able to reduce the problem to study properties of the graph when in stationarity---something that is often more tractable, and is indeed exploited in Section \ref{HittingTimes}.

	\subsection{Hitting times.}
	In Section \ref{HittingTimes} we present a result on the expected time it takes for the graph, starting with $j$ edges, to reach a fixed number $i$ edges, where $j<i$. We give an explicit expression for the expected value of this time, as well as study its asymptotic properties. 
	Special care is given to the case when $j=0$, i.e.\ when the dynamic graph starts without any edges.
	
	Let $\eta(t)$ denote the number of edges at time $t$ in $\{G(t),\ t\geq 0 \}$. Then $\{\eta(t), t\geq 0\}$ is a birth and death process on the non-negative integers $\{0,1,\ldots, N\}$, with birth rates $\lambda_k = (N-k) \beta/(n-1)$ and death rates $\mu_k = k \alpha$. This is also known as an (asymmetric) Ehrenfest urn model, or a dog flea model.
	
	For such processes it is well known that the hitting time of $i$ is distributed as the sum of $i$ independent exponentially distributed random variables, whose parameters are given by the nonzero eigenvalues of the matrix $-Q$, where $Q$ is the generator matrix of the birth and death process $\{\eta(t), t\geq 0\}$, with $i$ turned into an absorbing state. So, $Q$ is given through $Q_{00} = -\lambda_0$, $Q_{01} =\lambda_0$ and for $k = 1,2,\ldots,i-1$:  $Q_{kk} = -(\lambda_k +\mu_k)$,  $Q_{k,k+1} = \lambda_k$ for $k=0,1, \ldots, i-1$ and $Q_{k,k-1} = \mu_k$, while all other elements of $Q$ are 0. (see e.g. \cite[Thm.\ 1.1]{Fill09}).
	
	Because the eigenvalues of a matrix are typically hard to find, we use another approach in deriving the expected hitting time of $i$.
	Still, the obtained expression is difficult to compute for large $n$, see Proposition \ref{hittingtheorem}, and therefore we also give bounds on the expected time it takes for the dynamic graph to go from $0$ to $i=  [c  n]$ edges, where $c$ is a constant and $[x]$ denotes the closest integer to $x$. The main reason this particular scaling is studied is its connection to the size of the largest component. Namely, if $\mathcal{G}(n,i(n))$ is a static Erd\H{o}s-R\'enyi graph with a prescribed number of edges and $|C(n,i(n))|$ is the size of the largest component of such a graph, it is possible to show that
	for every $0 < \epsilon < 1$ there exist a $c_{\epsilon}>1/2$ such that if $i(n) = [c_{\epsilon}  n]$ then, 
	\[
	\frac{|C(n,i(n))|}{n} \xrightarrow[]{p} \epsilon \text{ as } n\to \infty,
	\]
	where $\xrightarrow[]{p}$ denotes convergence in probability. Furthermore,
	\[
	c_{\epsilon} = \frac{- \log (1-\epsilon)}{2  \epsilon}.
	\]
	Hence, for given $\epsilon \in (0,1)$ we know how many edges are needed in the static Erd\H{o}s-R\'enyi graph for the fraction of vertices in  the largest component to be roughly equal to $\epsilon$ with high probability, namely $i = [c_{\epsilon}  n]$ where $c_{\epsilon} = \frac{- \log (1-\epsilon)}{2  \epsilon}$. This can be used for the dynamic graph; if we wait until that many edges are present it is very likely that the size of the largest component in the dynamic graph has already exceeded $\epsilon  n$. This will be discussed further in Section \ref{sizecomponent}.
	
	The expected time to go from $0$ to $i = [c  n]$ exhibits three different behaviors depending on the value of $c$ (note that the expected number of edges grows as $\frac{\beta}{2 \alpha} n$).
	For $c < \frac{\beta}{2 \alpha}$ (i.e.\ for $i$ less than the asymptotic---in time---expected value of $\eta(t)$)  the graph reaches $i$ edges after roughly a constant time; 
	for $c = \frac{\beta}{2 \alpha}$ (i.e.\ for $i$ equal to the asymptotic expected value of $\eta(t)$) the graph reaches $i$ edges after an logarithmic time, which follows from the time to stationarity being $O(\log ( n))$;
	while for $c>\frac{\beta}{2 \alpha}$ (i.e.\ for $i$ greater than the asymptotic expected value of $\eta(t)$) the graph reaches $i$ edges after en exponentially large time. In Section \ref{HittingTimes} we prove the following:
	\begin{theorem}
		\label{mainHittingTimes}
		Let $\tau_j(i)$ be the time it takes, starting with $j$ edges, for the dynamic Erd\H{o}s-R\'enyi graph to reach $i = [c  n]$ edges, where $c>0$. Then for $n \to \infty$,
		\begin{enumerate}
			\item[(a)] If $c < \frac{\beta}{2\alpha}$ then,
			\begin{align*}
				&\tau_0(i) \xrightarrow[]{p} \frac{-\log(1-\frac{2\alpha}{\beta}c)}{\alpha} \text{ as } n \to \infty, \\ 
				&\mathbb{E}(\tau_0(i)) \to  \frac{-\log(1-\frac{2\alpha}{\beta}c)}{\alpha}  \text{ as } n \to \infty.
			\end{align*}
			
			\item[(b)] If $c = \frac{\beta}{2\alpha}$ then for all $j \in \mathbb{N}$,
			\begin{align*}
				\mathbb{E}(\tau_j(i)) = O(\log(n)).
			\end{align*}
			\item[(c)] If $c > \frac{\beta}{2\alpha}$ then
			\begin{align*}
				\Theta(n^{-1}) 
				e^{n\left( c\log(\frac{2\alpha}{\beta}c)-c+\frac{\beta}{2\alpha}\right)}
				\leq \mathbb{E}(\tau_0(i)) \leq \Theta(n^{-1/2}) 
				e^{n\left( c\log(\frac{2\alpha}{\beta}c)-c+\frac{\beta}{2\alpha}\right)}
			\end{align*}
			where $c \log(\frac{2\alpha}{\beta} c)-c+\frac{\beta}{2\alpha}>0$. Furthermore, $\frac{\tau_0(i)}{\mathbb{E}(\tau_0(i))}$ converges in distribution to an exponential random variable with expectation 1.
		\end{enumerate}
	\end{theorem}
	
	Theorem \ref{mainHittingTimes} may be used to provide bounds for the expected time the dynamic Erd\H{o}s-R\'enyi graph needs to first contain a component of a desired size. In particular, in the critical case ($\alpha = \beta$)---for which the typical size of the largest cluster is $o(n)$---we can find an upper bound for the expected time until a positive fraction (say $\epsilon$, where $\epsilon$ is small enough) of the vertices is in one connected component:
	
	\begin{corollary}
		\label{compsizecor}
		Let $\hat{\tau}(\epsilon n)$ be the first time  the critical dynamic Erd\H{o}s-R\'enyi graph, starting with no edges, has a component of size at least $\epsilon n$. Then for all  $\hat{\epsilon} >\epsilon$,
		\[
		\mathbb{E}[\hat{\tau}(\epsilon  n)] = O(n^{-1/2})e^{n(\hat{\epsilon}^2/16+O_{\hat{\epsilon}}(\hat{\epsilon}^3))}.
		\]
	\end{corollary}

	\begin{remark}
		The bound of Corollary \ref{compsizecor} is obtained by analyzing the time until the critical dynamic Erd\H{o}s-R\'enyi graph contains enough edges to make a large enough component probable. It is also possible that a component of large enough size appears when the number of edges is not quite as high as needed for our argument in the Corollary, but because a lower number of edges actually form an unlikely configuration.
		
		Indeed, while finishing this paper we were made aware of a result by O'Connell \cite{OCon98} which, with some extra work, provides us with sharper bounds on $\mathbb{E}[\hat{\tau}(\epsilon  n)]$:  
		$$\mathbb{E}[\hat{\tau}(\epsilon  n)] \leq  e^{n (\epsilon^3/8+  o_{\epsilon}(\epsilon^3)) + o_n(n) }.$$
		
		In Section \ref{sizecomponent} we prove this bound as well as show that it is indeed sharper than the bound in Corollary \ref{compsizecor}---showing that, in the critical case, a large component is more likely to emerge with fewer edges than is expected for a critical Erd\H{o}s-R\'enyi graph. The reason for studying the critical case of the dynamic graph is because of tractability of formulae. With that said, we still provide Corollary \ref{compsizecor}, because we believe that its proof is insightful in itself.
	\end{remark}


	

	Finally we provide a brief outline of the paper. In Section \ref{FTTS} we prove Theorem \ref{Fastesttimetostationary}---regarding the fastest time to stationarity---as well as study some asymptotic properties of its distribution function. In Section \ref{HittingTimes} we prove Theorem \ref{mainHittingTimes}, as well as give an explicit expression for the expected time it takes for the graph to go from $i$ to $j$ edges. We also discuss how the time until a graph component of given size emerges relates to the hitting time of a appropriately chosen number of edges.

	\section{The fastest time to stationarity}
	\label{FTTS}
	In constructing the fastest time to stationarity for $\{G(t),\ t\geq 0 \}$ we shall find the fastest times $\{T_{u,v}\}$ to stationary for the underlying edge processes $\{\chi_{u,v}(t),\ t\geq 0 \}$, and take $T_s$ to be the maximum of these.
	Waiting until all the edge processes have entered stationarity should ensure that the dynamic graph is in stationarity, since $G(t) = (\chi_{1,2}(t),\ldots, \chi_{n-1,n}(t))$. In order to show that this time to stationarity is indeed the fastest time to stationarity we need the concepts of a \textit{strong stationary time} and of \textit{separation}, as defined in \cite{Fill}.
	\subsection{Separation and strong stationary times}
	\label{sepAndSST} 
	Roughly speaking, a strong stationary time $T$ for a stochastic process $X$ is a stopping time for $X$ with some extra external randomness such that $X(T)$ has the stationary distribution and is independent of $T$.
	In order to define a strong stationary time for a process $X$ one needs the concept of a \textit{randomized stopping time}. 
	
	\begin{definition} ({\cite{Fill}})
		\label{randomizedstoppingtime}
		Let $(\Omega, \mathcal{F},\{ \mathcal{F}_t \}, P)$ be a filtered probability space. Let $\mathcal{F}_{\infty}$ be the smallest $\sigma$-algebra containing $\mathcal{F}_t$ for all $t$.\\
		Furthermore, let $\mathcal{G} \subset \mathcal{F}$ be a sub-$\sigma$-algebra of $\mathcal{F}$ independent of $\mathcal{F}_{\infty}$.  We say that $T:\Omega \to [0,\infty]$ is a randomized stopping time relative to $\{ \mathcal{F}_t,\ t\geq 0 \}$ if for each $t\geq 0$,
		\[
		\{T \leq t\} \in \sigma (\mathcal{F}_t, \mathcal{G}),
		\]
		where $\sigma (\mathcal{F}_t, \mathcal{G})$ is the smallest $\sigma$-algebra containing both $\mathcal{F}_t$ and $\mathcal{G}$.\\
		If the process $X$ is adapted to $\{ \mathcal{F}_t,\ t\geq 0 \}$ we say that $T$ is a randomized stopping time for $X$.
	\end{definition}
	
	We are now ready to define the strong stationary time and the fastest time to stationarity.
	\begin{definition}({\cite{Fill}})
		\label{strongstationarytime}
		Let $X$ be a stochastic process, defined on a filtered probability space $(\Omega, \mathcal{F},\{ \mathcal{F}_t,\ t\geq 0 \}, P)$ and adapted to $\{ \mathcal{F}_t,\ t\geq 0 \}$, taking values in some state space $S$. Assume that $X$ has a unique stationary distribution $\pi$.
		Furthermore let $T$ be a randomized stopping time relative to $\{ \mathcal{F}_t,\ t\geq 0 \}$. Then, $T$ is said to be a \textit{strong stationary time} for $X$ if: $X(T)$ has the stationary distribution and is independent of $T$ given that $\{T < \infty\}$, i.e. if,
		\begin{multline*}
			\mathbb{P}(T\leq t, X(T) = y|T<\infty) = \mathbb{P}(T\leq t|T<\infty) \mathbb{P}(X(T)=y|T<\infty)\\ = \mathbb{P}(T\leq t|T<\infty) \pi(y)
		\end{multline*}
		for all $0\leq t <\infty$ and $y \in S$.
		
		If, for any other strong stationary time $T'$, we have
		$
		\mathbb{P}(T>t)\leq \mathbb{P}(T'>t)
		$
		then we say that $T$ is the \textit{fastest time to stationarity.}
	\end{definition}
	\begin{remark}
		We shall only be concerned with strong stationary times $T$ such that $\mathbb{P}(T<\infty) = 1$, hence we can drop the conditioning on $\{T < \infty\}$ in Definition \ref{strongstationarytime} above.
	\end{remark}

	The separation, $s(t) = \sup_{y} \left( 1-\frac{\mathbb{P}(X(t) = y)}{\pi(y)}\right) $, for a stochastic process is a function in time which measures the ``distance'' between the distribution at time $t$ and its stationary distribution, and has an intimate connection with strong stationary times.
	
	Strong stationary times are well-understood for ergodic Markov processes on countable state spaces, see \cite{Fill}. 
	The main result of \cite{Fill} is that for ergodic Markov processes on countable state spaces the following holds,
	
	\begin{enumerate}[(I)]
		\item If $T$ is a strong stationary time, then for all $0\leq t <\infty$, 
		\begin{equation}
			s(t) \leq \mathbb{P}(T>t) \label{separationeq}
		\end{equation}
		i.e. the separation at time $t$ is a lower bound for the probability that the process has not yet entered stationarity at time $t$.
		\item If the state space of the process is finite (and the underlying probability space rich enough to support an uniformly distributed random variable on $(0,1)$ independent of the process), there exist a strong stationary time $T$ such that \eqref{separationeq} holds with equality. We call such a time  \textit{the fastest time to stationarity.}\label{separationequality}
	\end{enumerate}
	\begin{remark}
		It follows that, if the distribution of the fastest time to stationarity is known, then equation \eqref{separationeq} gives a way of quantifying the rate of convergence of the dynamic graph to stationarity, since the separation measures the distance between the distribution of a process at time $t$ and its stationary distribution. \\
		As stated before, we assume that the underlying probability space is rich enough to support (iia) in the definition of $\{G(t),\ t\geq 0 \}$---for the purpose of finding the distribution of the fastest time to stationarity this assumption can be made without loss of generality, since if such a time exists on the probability space its distribution is determined by (II).
	\end{remark}
	
	Fill also gives the algorithm for constructing the fastest time to stationarity for an ergodic Markov chain on a finite state space. 
	However, the proof of this is technical and not very intuitive. Nevertheless, we can still use the above mentioned results from \cite{Fill} to show that our candidate time to stationarity for $\{G(t),\ t\geq 0 \}$ is indeed the stochastically smallest one. 
	
	\subsection{The fastest time to stationarity}
	In order to construct the fastest time to stationarity for the dynamic graph we proceed by constructing such times for the underlying edge processes.
	
	\begin{lemma}
		\label{edgetimes}
		Let $\{\chi_{u,v}(t),\ t\geq 0 \}$ be an edge process starting with $0$ or $1$ edges. Then the fastest time to stationarity $T_{u,v}$ for $\{\chi_{u,v}(t),\ t\geq 0 \}$ is distributed as,
		\[
		T_{u,v} \sim \text{Exp$\left( \alpha+\frac{\beta}{n-1}\right)$ }.
		\]
	\end{lemma}
	\begin{proof}
		After the time to the first update $T_{u,v}$ of an edge---which is an exponential distributed time with rate $(\alpha + \frac{\beta}{n-1})$---the edge process is in stationarity, since 
		$\mathbb{P}(\chi_{u,v}(T_{u,v}+s)=1)=\mathbb{P}(\text{edge added at last update}) = p$. It is also clear that $T_{u,v}$ is a (randomized) stopping time relative to the filtration which $\{\chi_{u,v}(t),\ t\geq 0 \}$ is adapted to (information generated by update times and corresponding edge updates).
		
		Also, $T_{u,v}$ satisfies,
		\begin{multline*}
			\mathbb{P}(T_{u,v}\leq t, \chi_{u,v}(T_{u,v}) = 1) = \mathbb{P}(T_{u,v}\leq t,\text{edge added last update}) \\
			= \mathbb{P}(T_{u,v}\leq t) \mathbb{P}(\text{edge added last update}) 
			= \mathbb{P}(T_{u,v}\leq t) p.
		\end{multline*}
		By Definition \ref{strongstationarytime}, $T_{u,v}$ is a strong stationary time for $\{\chi_{u,v}(t),\ t\geq 0 \}$. 
		
		Furthermore, it is easily shown that $\mathbb{P}(T_{u,v}>t)$ equals the separation\[s(t) = \sup_{i \in \{0,1\}}\left( 1-\frac{\mathbb{P}(\chi_{u,v}(t)=i)}{\pi (i)}\right) \] of the process $\{\chi_{u,v}(t),\ t\geq 0 \}$ since,
		\begin{align*}
			& \chi_{u,v}(0)=0 \implies s(t) = \left(1-\frac{p_{0,1}(t)}{\pi(1)}\right) = \mathbb{P}(T_{u,v}>t),\\
			& \chi_{u,v}(0)=1 \implies s(t) = \left(1-\frac{p_{1,0}(t)}{\pi(0)}\right) = \mathbb{P}(T_{u,v}>t).
		\end{align*}
		Hence, by (II), $T_{u,v}$ is the fastest time to stationarity for  $\{\chi_{u,v}(t),\ t\geq 0 \}$, if the process starts with $0$ or $1$ edges.
	\end{proof}
	\begin{remark}
		\label{remstat}
		If the initial distribution of $\{\chi_{u,v}(t),\ t\geq 0 \}$ is arbitrary then the time in Lemma \ref{edgetimes} is still a strong stationary time for the process---however it need not be the fastest one: clearly, if the initial distribution is the stationary distribution then the fastest time to stationarity is $T = 0$.
	\end{remark}
	
	We are now ready to prove Theorem \ref{Fastesttimetostationary}.
	\begin{proof}[Proof of Theorem \ref{Fastesttimetostationary}]
		Recall that $G(t) = (\chi_{1,2}(t),\dots, \chi_{n-1,n}(t))$. \\
		Assume $\{G(t),\ t\geq 0 \}$ starts in a given state $\sigma$ in the state space $S$---i.e. $G(0)$ has probability mass 1 on $\sigma$. By Lemma \ref{edgetimes}, $T_{u,v} \sim $ Exp($\alpha + \frac{\beta}{n-1}$) is the fastest time to stationary for $\{\chi_{u,v},\ t\geq 0 \}$ whether we start with or without an edge. 
		It is clear from the defining property (iia) that $T_s = \max(T_{1,2},\dots,T_{n-1,n})$ is a strong stationarity time for $\{G(t),\ t\geq 0 \}$.
		
		It remains to show then is that $P(T_s > t) = s(t)$ where $s(t)$ is the separation for $\{G(t),\ t\geq 0 \}$ starting in state $\sigma$. Note that $P(T_s > t) = s(t)$ is equivalent to $P(T_s < t) = 1-s(t) = a(t) = \inf_{x \in S} \frac{\mathbb{P}( G(t)=x|G(0)=\sigma  )}{\pi(x)}$, which is easier to prove ($\pi$ being the stationary distribution of $\{G(t),\ t\geq 0 \}$). Let $a_{u,v}(t) = 1-s_{u,v}(t)$ where $s_{u,v}(t)$ is the separation for the edge process $\{\chi_{u,v}(t),\ t\geq 0 \}$ starting in state $0$ or $1$ (the separation is the same for both states). As the edge processes are independent we have that $a(t) = a_{u,v}(t)^N$. 
		Recall that since $T_{u,v}$ is the fastest time to stationarity for $\{\chi_{u,v},\ t\geq 0 \}$ we have that $P(T_{u,v}<t)=1-s_{u,v}(t)=a_{u,v}(t)$. We get,
		\begin{align*}
			a(t) = a_{u,v}(t)^N = (1-s_{u,v}(t))^N= \mathbb{P}(T_{u,v}<t)^N=\left( 1 -  e^{-(\alpha+\frac{\beta}{n-1}) t}\right)  ^{N}
		\end{align*}
		Since $\mathbb{P}(T_s \leq t) = \mathbb{P}( \max(T_{1,2},\dots,T_{n-1,n}) \leq t )
		= \mathbb{P}(T_{u,v} \leq t)^{N}$ the assertion follows.
	\end{proof}
	\begin{remark}
		If the initial distribution of $\{G(t),\ t\geq 0 \}$ is arbitrary---not a fixed number of edges---then the time in Lemma \ref{Fastesttimetostationary} is still a strong stationary time for the process, however it need not be the fastest one (see Remark \ref{remstat}). 
	\end{remark}

	\subsection{Asymptotics}
	The distribution for the fastest time to stationarity in Theorem \ref{Fastesttimetostationary} is exact but not very insightful. 
	Here we provide the proof of Corollary \ref{asymptoticlemma}, in which we deal with the asymptotic of the distribution function. The corollary follows from standard results in extreme value theory, but since the proof is not difficult we provide it here for reasons of completeness.

	\begin{proof}[Proof of Corollary \ref{asymptoticlemma}]
		Let all limits be for $n \to \infty$. Recall for $a \in \mathbb{R}$, the standard limit $(1-a/n)^n \to e^{-a}$ (and therefore $(1-(af(n))/n)^n \to e^{-a}$, if $f(n) \to 1$).
		In particular using Theorem \ref{Fastesttimetostationary},
		\begin{multline*}	
			\mathbb{P}\left(T_s  \leq  \frac{\log N  +y}{\alpha}\right)
			= \left( 1 -  e^{-(\alpha+\frac{\beta}{n-1})  \frac{\log N  +y}{\alpha}}\right)^{N} \\
			=  \left( 1 - \frac{e^{-y - \frac{\beta}{n-1}\frac{\log N  +y}{\alpha}}}{N}\right)^{N} \to e^{-e^{-y}}, 
		\end{multline*}
		since $\frac{\beta}{n-1}\frac{\log N  +y}{\alpha} \to 0$ and $N \to \infty$ as $n\to \infty$.
		Observing that $$\log N = 2 \log(n) + \log(1-1/n) - \log(2)$$ and replacing $y$ by $x + \log(1-1/n) = x + O(1/n)$ shows that 
		$$\mathbb{P}\left(\alpha T_s - 2 \log(n) + \log 2  \leq  x\right) \to e^{-e^{-x}},$$ as desired.
		
		To prove that $\mathbb{E}(T_s) = O(\log (n))$ we note that,
		$$
		1-(1-e^{-(\alpha+ \frac{\beta}{n-1} ) t })^N \leq  \min \{1,N  e^{ -(\alpha+ \frac{\beta}{n-1} ) t  } \}\qquad \text{and} \qquad \log(N)  \leq 2\log(n).
		$$
		Hence,
		\begin{multline*}
			\mathbb{E}(T_s) = \int_0^{\infty} \mathbb{P}(T_s > t) dt =
			\int_0^{\infty} \left(1- \left( 1 -  e^{-(\alpha+\frac{\beta}{n-1}) 
				t}\right)  ^{N} \right) dt\\
			= \int_0^{\log(N)/\alpha} \left(1- \left( 1 - 
			e^{-(\alpha+\frac{\beta}{n-1}) t}\right) ^{N} \right) dt +
			\int_{\log(N)/\alpha}^{\infty} \left(1- \left( 1 - 
			e^{-(\alpha+\frac{\beta}{n-1}) t}\right) ^{N} \right) dt\\
			\leq \frac{\log(N)}{\alpha} + \int_{\log(N)/\alpha}^{\infty} N 
			e^{-(\alpha+\frac{\beta}{n-1}) t} dt \leq 2\frac{\log(n)}{\alpha}  + 
			\left(\alpha+\frac{\beta}{n-1}\right)^{-1} = 
			O(\log(n)).
		\end{multline*}
		This completes the proof of the corollary.	
	\end{proof}

	\section{Hitting times for a fixed number of edges}
	\label{HittingTimes}
	In this section we study $\mathbb{E}(\tau_j(i))$, where $\tau_j(i)$ is the time it takes for the dynamic graph to reach $i$ edges, given that it starts with $j$ edges, i.e.
	$$\tau_j(i) = \inf \{t>0;\ \eta(t)=i,\ \eta(0)=j \}.$$ We derive an exact expression for $\mathbb{E}(\tau_j(i))$, $(j<i)$, and give special attention to the case $\mathbb{E}(\tau_0(i))$, when $i=[c  n]$ and $c>0$, where we provide (asymptotic) bounds for $\mathbb{E}(\tau_0(i))$.
	
	As mentioned in the introduction, it is known (see e.g.\ \cite{Fill09}), that $\tau_0(i)$ is distributed as the sum of $i$ independent exponentially distributed random variables with as rate parameters, the nonzero eigenvalues of the negative generator matrix of the variant of the process $\{\eta(t), t \geq 0\}$ restricted to states $\{0, 1, \ldots, i\}$ in which $i$ is turned into an absorbing state. However, finding those eigenvalues is difficult, and therefore we put effort in finding expressions for the expected hitting time of $i$ and asymptotics for it.
	
	In deriving an exact expression for $\mathbb{E}(\tau_j(i))$ we shall need to exploit the strong Markov property of the dynamic graph. For our purposes we say that a Markov process $X$ has the strong Markov property if for any $\textit{a.s. finite}$ stopping time $\tau$ for $X$ we have that $X_{\tau} = \{X(t+\tau),\ t\geq 0 \}$ is a probabilistic copy of $X$ starting in $X(\tau)$, as well as being independent of $X$ up to time $\tau$, given $X(\tau)$.
	
	First we compute $\mathbb{E}(\tau_k(k+1))$ for $k<N$. Then, by the strong Markov property of $\{\eta(t),\ t\geq 0 \}$,
	\begin{equation}
		\label{StrongMarkovProperty}
		\mathbb{E}(\tau_j(i)) = \sum_{k=j}^{i-1} \mathbb{E}(\tau_k(k+1)).
	\end{equation}
	This leads us to the following proposition.
	\begin{proposition}
		\label{hittingtheorem}
		Let $\tau_j(i)$ be the time it takes for the dynamic Erd\H{o}s-R\'enyi graph, starting with $j$ edges, to reach $i$ edges, where $j<i$. Then,
		\begin{equation}
			\label{eq1}
			\mathbb{E}(\tau_i(i+1)) = \frac{(n-1) (N-i-1)!i!}{\beta  N!} \sum_{k=0}^{i} \binom{N}{i-k}\left(\frac{\alpha}{\beta}(n-1)\right)^k
		\end{equation}
		and 
		\begin{equation}
			\label{eq3}
			\mathbb{E}(\tau_j(i)) =  \sum_{m=j}^{i-1} \frac{(n-1) (N-m-1)!m!}{\beta N!}  \sum_{k=0}^{m} \binom{N}{m-k}\left(\frac{\alpha}{\beta}(n-1)\right)^k.
		\end{equation}
		
	\end{proposition}
	\begin{proof}
		Recall that $\{\eta (t),\ t\geq 0 \}$ is an ergodic Markov chain on a finite state space, this ensures that the process has the strong Markov property, see \cite[Thm.~4.1]{SMP}. For notational convenience let $\lambda_k = (N-k)\beta/(n-1)$ be the birth rate and $\mu_k = \alpha k$ be the death rate in state $k$ of $\{\eta (t),\ t\geq 0 \}$.
		
		We begin by deriving a recursive formula for $\mathbb{E}(\tau_i(i+1))$. Since $\{\eta (t),\ t\geq 0 \}$ is ergodic and therefore positively recurrent we have that $\mathbb{E}(\tau_i(i+1))<\infty$.
		We derive a recursive formula for $\mathbb{E}(\tau_i(i+1))$ by conditioning on the first jump. Let $p(i,i+1) = \frac{\lambda_i}{\lambda_i +\mu_i}$ be the probability that the process moves from $i$ edges to $i+1$ edges, and let $ i \to (i+1) $ indicate such an event. Define $p(i,i-1) = \frac{\mu_i}{\lambda_i +\mu_i}$ and  $i \to (i-1) $ in an analogous way. Also let $H_i\sim$ Exp($\lambda_i+\mu_i$) be the holding time in state $i$.
		Then,
		\begin{align*}
			\mathbb{E}(\tau_i(i+1)) & =  p(i,i+1) \mathbb{E}(\tau_i(i+1)|i \to (i+1)) + p(i,i-1) \mathbb{E}(\tau_i(i+1)|i \to (i-1))\\
			& = p(i,i+1) \mathbb{E}(H_i) + p(i,i-1) \mathbb{E}(\tau_i(i+1)|i \to (i-1))\\
			& \overset{(i)}{=} p(i,i+1) \mathbb{E}(H_i)+p(i,i-1) \mathbb{E}(H_i) + p(i,i-1) \mathbb{E}(\tau_{i-1}(i+1))\\
			& \overset{(ii)}{=} \mathbb{E}(H_i) + p(i,i-1)( \mathbb{E}(\tau_{i-1}(i)) +\mathbb{E}(\tau_{i}(i+1)) ).
		\end{align*}
		For $(i)$ we used that when entering state $i-1$ the process probabilistically restarts itself, this is by the strong Markov property as well as $\tau_i(i-1)$ being a stopping time for $\{\eta (t),\ t\geq 0 \}$ starting in $i$. Equality
		$(ii)$ follows since, 
		\[
		\tau_{i-1}(i+1) = \tau_{i-1}(i)+\tau_{i}'(i+1)
		\]
		where $\tau_{i}'(i+1)$ is the time it takes for the process, starting with $i-1$ edges, to go from $i$ edges (when it eventually reaches $i$ edges) to $i+1$ edges. This is then, again by the strong Markov property, distributed as $\tau_{i}(i+1)$. Hence,
		\[
		\mathbb{E}(\tau_{i-1}(i+1)) = \mathbb{E}(\tau_{i-1}(i))+\mathbb{E}(\tau_{i}(i+1)).
		\]
		
		Solving the above equation for $\mathbb{E}(\tau_{i}(i+1))$ we obtain,
		\begin{equation}
			\label{eq2}
			\mathbb{E}(\tau_i(i+1) )= \frac{\mathbb{E}(H_i)+p(i,i-1) \mathbb{E}(\tau_{i-1}(i))}{p(i,i+1)} ,\qquad i \in \{1,2,\dots, N-1\}.
		\end{equation}
		For $i=0$ we have $\mathbb{E}(\tau_0(1)) = \mathbb{E}(H_0)$.
		
		To prove \eqref{eq1} we use \eqref{eq2} together with induction.
		The following holds for the birth-death process, $\{\eta (t),\ t\geq 0 \}$.
		\begin{align*}
			& \mathbb{E}(H_i) = \frac{1}{\lambda_i+\mu_i} = \frac{ n-1 }{ (N-i)\beta+(n-1) i \alpha  } \\
			& p(i,i-1) = \frac{\mu_i}{\lambda_i+\mu_i} = \frac{ i (n-1)\alpha }{ (N-i)\beta+(n-1) i \alpha } \\
			& p(i,i+1) = \frac{\lambda_i}{\lambda_i+\mu_i} = \frac{ (N-i)\beta }{ (N-i)\beta+(n-1) i \alpha }
		\end{align*}
		Inserting this in \eqref{eq2}, we obtain
		\[
		\mathbb{E}(\tau_i(i+1)) = \frac{n-1}{(N-i)\beta} + \frac{(n-1) i}{N-i} \frac{\alpha}{\beta}  \mathbb{E}(\tau_{i-1}(i)),\qquad i \in \{1,2,\dots, N-1\}.
		\]
		For $i=0$, we have by \eqref{eq1} that
		$
		\mathbb{E}(\tau_0(1)) = \frac{n-1}{\beta  N},
		$
		which is indeed equal to $\mathbb{E}(H_0)$.
		
		Assume that \eqref{eq1} holds for arbitrary $ i < N-1 $. Then,
		\begin{align*}
			& \mathbb{E}(\tau_{i+1}(i+2)) = \frac{\mathbb{E}(H_{i+1}) + p(i+1,i)\mathbb{E}(\tau_i(i+1))}{p(i+1,i+2)}
			\\
			& =\frac{n-1}{(N-i-1)\beta} + \frac{(n-1) (i+1)}{N-i-1} \left(\frac{(n-1)(N-i-1)!i!}{\beta  N!} \sum_{k=0}^{i} \binom{N}{i-k} \left((n-1)\frac{\alpha}{\beta}\right)^k\right)
		\end{align*}
		which after standard, but tedious algebra, equals
		\[
		\frac{(n-1)(N-i-2)!(i+1)!}{\beta  N!} \sum_{k=0}^{i+1} \binom{N}{i+1-k}\left((n-1)\frac{\alpha}{\beta}\right)^k.
		\]
		This proves equation \eqref{eq1}, and \eqref{eq3} follows from \eqref{StrongMarkovProperty}.
		
	\end{proof}
	
	The expression \eqref{eq3} is exact but not very insightful. In order to get a better understanding of how $\mathbb{E}(\tau_j(i))$ behaves, we study how this expectation grows when $i = [c  n]$ and $j=[c' n]$.  
	To do this for the case $c'<c<\frac{\beta}{2\alpha}$, consider $\bar{\eta}(t)= \frac{\eta(t)}{n}$ and note that $\bar{\eta}(t)$ increases by $\frac{1}{n}$ at rate $\beta \frac{n}{2} -\beta \frac{n}{n-1}\bar{\eta}(t)$ and decreases by $\frac{1}{n}$ at rate $\alpha n \bar{\eta}(t)$. This implies that a candidate for a deterministic approximation of $\bar{\eta}(t)$ satisfies 
	$$\frac{d \bar{\eta}(t)}{dt} =  \frac{\beta}{2} -\frac{\beta \bar{\eta}(t)}{n-1} - \alpha \bar{\eta}(t) \qquad \text{and} \qquad \bar{\eta}(0) = \frac{[c'n]}{n}.$$ As $n \to \infty$ this reads 
	$$\frac{d \bar{\eta}(t)}{dt} =  \frac{\beta}{2} - \alpha \bar{\eta}(t)\qquad \text{and} \qquad\bar{\eta}(0) = c'.$$  
	This differential equation is solved by $\bar{\eta}(t) = \frac{\beta}{2 \alpha} (1-e^{-\alpha t})+ c'e^{-\alpha t}$. 
	
	Now we are ready to formulate our next lemma.
	
	\begin{lemma}
		\label{stationaryedgeslow}
		Let $\tau_j(i)$ be the time it takes for the dynamic Erd\H{o}s-R\'enyi graph, starting with $j =[c'n]$ edges, to reach $i = [c   n]$ edges, where either $0 \leq c'<c<\frac{\beta}{2\alpha}$ or 
		$\frac{\beta}{2\alpha}<c<c'$. Then,
		$$\tau_j(i)  \xrightarrow[]{p} \frac{-\log\left(\frac{\beta -2 \alpha   c}{\beta -2 \alpha  c'}\right)}{\alpha} \qquad \text{as $n \to \infty$.} $$
	\end{lemma}
	\begin{proof}
		Let all limits be for $n \to \infty$.
		In this proof we use $\{\bar{\eta}^{(n)}(t), t\geq 0\}$ and $\{\eta^{(n)}(t), t\geq 0\}$ to denote the dependence on the number of vertices in the graph. 
		
		We prove that $\{\bar{\eta}^{(n)}(t), t\geq 0\}$ converges pointwise in distribution to the deterministic process $\{\frac{\beta}{2 \alpha} (1-e^{-\alpha t})+c'e^{-\alpha t},t\geq 0\}$ as $n\to \infty$. That is, for given $t$,
		$$\bar{\eta}^{(n)}(t) \to \frac{\beta}{2 \alpha} (1-e^{-\alpha t}) + c'e^{-\alpha t} \qquad \text{in distribution.}$$
		We now distinguish between the set of $j$ edges which are present at time 0 and the set of $N-j$ edges which are not present at time 0. 
		Note that $\eta^{(n)}(t)$ is then the sum of two independent binomial distributed random variables one with parameters $N-j$ and $p_0^{(n)}(t) = \frac{\beta}{\beta + (n-1)\alpha}(1-e^{-(\alpha + \frac{\beta}{n-1})t})$
		and the other with parameters $j$ and $p_1^{(n)}(t) = \frac{\beta}{\beta + (n-1)\alpha}(1-e^{-(\alpha + \frac{\beta}{n-1})t}) + e^{-(\alpha + \frac{\beta}{n-1})t}$. 
		So we obtain that the moment generating function (MGF) $$M^{(n)}(\sigma)= \mathbb{E}[e^{\sigma \bar{\eta}^{(n)}(t)}] = \mathbb{E}[e^{\frac{\sigma}{n} \eta^{(n)}(t)}]$$ is given by 
		$$M^{(n)}(\sigma) = \left(1- p_0^{(n)}(t)(1-e^{\sigma/n})\right)^{N-j}\left(1- p_1^{(n)}(t)(1-e^{\sigma/n})\right)^{j}.$$
		Noting that $1-e^{\sigma/n} = -\sigma/n + o(1/n)$, $p_0^{(n)}(t) = \frac{\beta (1-e^{-\alpha t})}{\alpha (n-1)} + o(1/n)$ and $p_1^{(n)}(t) =  e^{-\alpha t} + o(1)$, we obtain that 
		\begin{multline*}
			M^{(n)}(\sigma) = \left(1+ \frac{\sigma \beta (1- e^{-\alpha t})}{\alpha n(n-1)} + o(1/n^2)\right)^{N-j}
			\left(1 +\frac{\sigma e^{-\alpha t}}{n} + o(1/n)\right)^{j}
			\\
			= \left(1+ \frac{\sigma \beta  (1-e^{-\alpha t})}{2\alpha (N-j)}(1-j/N) + o(1/n^2)\right)^{N-j}\left(1 +\frac{\sigma e^{-\alpha t}}{n} + o(1/n)\right)^{j}.
		\end{multline*}
		We now distinguish between $j=0$ and $j = [c'n]$ for $0 < c'$.
		In the former case $M^{(n)}(\sigma) \to  e^{\sigma \frac{\beta  (1-e^{-\alpha t})}{2\alpha}}$, while in the latter case 
		\begin{multline*}
			M^{(n)}(\sigma) = \left(1+ \frac{\sigma \beta  (1-e^{-\alpha t})}{2\alpha (N-j)} + o(1/n^2)\right)^{N-j}\left(1 +\frac{\sigma e^{-\alpha t}}{j} c' + o(1/n)\right)^{j}\\ \to e^{\sigma \frac{\beta  (1-e^{-\alpha t})}{2\alpha}} 
			e^{\sigma c'  e^{-\alpha t}}
		\end{multline*}
		
		Those expressions are indeed the MGF of the constant $\frac{\beta}{2 \alpha} (1-e^{-\alpha t})+ c'e^{-\alpha t}$ for both $c' = 0$ and $c'>0$ and since convergence of MGFs implies convergence in distribution, we obtain pointwise convergence in distribution of the stochastic process to the deterministic process.
		
		We now focus on $c<\frac{\beta}{2\alpha}$. The  proof for $c>\frac{\beta}{2\alpha}$ is analogous.
		We deduced that for $0<c'<c < \frac{\beta}{2 \alpha}$ and  $ \displaystyle t_{c',c} = \frac{-\log\left(\frac{\beta -2 \alpha  c}{\beta -2 \alpha  c'}\right)}{\alpha}$, we have $ \bar{\eta}^{(n)}(t_{c',c}) \to c$ in distribution. Note that $t_{c,c'}>0$ indeed since $c>c'$. The above implies that if $t> t_{c',c}$, then $\mathbb{P}(\eta^{(n)}(t) \geq cn)$ is arbitrary close to 1 for large enough $n$. Since $\eta^{(n)}(t) \leq \max_{s \leq t} \eta^{(n)}(s)$, this immediately implies that $\mathbb{P}(\max_{s \leq t} \eta^{(n)}(s) \geq cn)$ is arbitrary close to 1 for large enough $n$.
		
		Similarly, for $t<t_{c',c}$, $\mathbb{P}(\eta^{(n)}(t) \geq cn)$ is arbitrary close to 0 for large enough $n$.
		Note that for $c<\frac{\beta}{2\alpha}$ and as long as $\eta(s) <c n$ (which is thus bounded above by $\frac{\beta}{2\alpha} n$), the process $\{\eta^{(n)}(t), t\geq 0\}$ has an upward drift. Hence if $s<t$ then $\mathbb{P}(\eta(t)\geq cn|\eta(s)\geq cn)\geq 1/2$. This implies that 
		$$\mathbb{P}(\eta^{(n)}(t) \geq cn|\max_{s \leq t} \eta^{(n)}(s) \geq cn)\geq \mathbb{P}(\eta(t)\geq cn|\eta(0)\geq cn)  \geq 1/2,$$
		which in turn implies that 
		$$\mathbb{P}(\max_{s \leq t} \eta^{(n)}(s) \geq cn) \leq  2 \mathbb{P}(\eta^{(n)}(t) \geq cn),$$
		which is arbitrary close to 0 for large enough $n$.
		This implies that for all $\epsilon>0$, we have $\mathbb{P}(|\tau_0([c n]) - t_{c',c}| >\epsilon) \to 0$, which proves convergence in probability. 
		
	\end{proof}
	Using Lemma \ref{stationaryedgeslow} we can prove Theorem \ref{mainHittingTimes} (a).
	\begin{proof}[Proof of  Theorem 2 (a) ]
		Let $i = [c  n]$ where $c < \frac{\beta}{2 \alpha}$ and set $t_c = t_{0,c}$.	
		The convergence in probability part is immediate, since by Lemma \ref{stationaryedgeslow} $\tau_0(i) \overset{p}{\to} t_c = \frac{-\log\left(1-\frac{2 \alpha}{\beta}c \right)}{\alpha} \text{ as } n\to \infty$. 
		Hence, for given $\epsilon>0$, we have that 
		$\mathbb{P}(\tau_0(i)  > t_c+\epsilon  ) <\epsilon$ and $\mathbb{P}(\tau_0(i)  > t_c-\epsilon  )>1-\epsilon$. Combining this with the fact that $\mathbb{E}(\tau_{j}(i)) \leq \mathbb{E}(\tau_0(i))$ for all $j < i$ and using the strong Markov property of $\{\eta(t),\ t\geq 0 \}$, we obtain that for given $\epsilon>0$ there exist $n_0$ such that for all $n>n_0$,
		\begin{multline*}
			\mathbb{E}(\tau_0(i)) \\
			= \mathbb{E}(\tau_0(i)  |  \tau_0(i)\leq t_c+\epsilon)\mathbb{P}(\tau_0(i)\leq t_c+\epsilon) + \mathbb{E}(\tau_0(i)  |  \tau_0(i)> t_c+\epsilon)\mathbb{P}(\tau_0(i)> t_c+\epsilon)\\
			\leq (t_c+\epsilon) + \epsilon (t_c+\epsilon+\mathbb{E}(\tau_{\eta(t_c+\epsilon)}(i)))\leq  t_c+\epsilon + \epsilon (t_c+\epsilon+\mathbb{E}(\tau_{0}(i)))
		\end{multline*}
		This implies that,
		$$\mathbb{E}(\tau_0(i)) \leq \frac{1+\epsilon}{1-\epsilon}(t_c+\epsilon).$$
		
		Using that for all $\epsilon>0$ and large enough $n$ we have $\mathbb{P}(\tau_0(i)  > t_c-\epsilon  )>1-\epsilon$. So, we get, by analogous calculations,
		\[
		\mathbb{E}(\tau_0(i)) \geq (t_c-\epsilon)(1-\epsilon).
		\]
		Together this implies that  $\mathbb{E}(\tau_0(i)) \to t_c$ as $n\to \infty$.

	\end{proof}
	
	\begin{remark}
		It follows with a little extra work from $\mathbb{E}[\tau_0(i)]<\infty$ and $\tau_0(i) \overset{p}{\to} t_c$  that $\tau_0(i)) \to t_c$ in expectation (see e.g.\ \cite[Problem 5.6.6]{Grim01}).
	\end{remark}

	The hitting time results for $c \geq \frac{\beta}{2\alpha}$, derived below provide us with bounds that are not tight, but they establish logarithmic growth for $c = \frac{\beta}{2\alpha}$; and exponential growth for $c>\frac{\beta}{2\alpha}$.
	
	\begin{proof}[Proof of Theorem 2 (b)]
		
		Let $T_s$ be the strong stationary time for the graph process $\{G(t),\ t\geq 0 \}$ as defined in Theorem \ref{Fastesttimetostationary}, and let $i = [\frac{\beta}{2\alpha}   n] $. We know that $E(T_{s}) = O(\log (n))$ and that $\eta(T_s)$ is binomially distributed with parameters $N$ and $\frac{\beta}{\beta + \alpha(n-1)}$, which implies that $\mathbb{E}(\eta(T_s)) = i + O(1)$. By the central limit theorem it follows that for large enough $n$, $\mathbb{P}(\eta(T_s)\leq i) \leq 2/3$. 
		Again observing that for all $0\leq j<i$, we have  $\mathbb{E}(\tau_j(i))\leq \mathbb{E}(\tau_0(i))$. Conditioning on whether $\{\eta(T_s)\geq i\}$  (equivalent to $\{\tau_0(i) \leq T_s\})$ or not gives,
		\begin{align*}
			\mathbb{E}(\tau_0(i)) & = \mathbb{E}(\tau_0(i)|\eta(T_s)\geq i) \mathbb{P}(\eta(T_s)\geq i) +  \mathbb{E}(\tau_0(i)|\eta(T_s)\leq i) \mathbb{P}(\eta(T_s)\leq i)	\\
			& = \mathbb{E}(T_s|\eta(T_s)\geq i) \mathbb{P}(\eta(T_s)\geq i) +  \mathbb{E}(T_s + \tau_{\eta(T_s)}(i)|\eta(T_s)\leq i) \mathbb{P}(\eta(T_s)\leq i)	\\
			&\leq  \mathbb{E}(T_s) + \frac{2}{3} (\mathbb{E}(\tau_{\eta(T_s)}(i)|\eta(T_s)<i) ) \\
			&\leq \mathbb{E}(T_s) + \frac{2}{3} \mathbb{E}(\tau_0(i)).
		\end{align*}	
		Which implies $\mathbb{E}(\tau_0(i)) \leq 3   \mathbb{E}(T_s) =  O(\log (n))$,
		as well as $\mathbb{E}(\tau_j(i)) = O(\log (n))$, for $j<i = [\frac{\beta}{2\alpha}   n]$.
		
		An analogous argument shows that $\mathbb{E}(\tau_j(i)) = O(\log (n))$ for $j>[\frac{\beta}{2\alpha}   n]$.
	\end{proof}
	
	Remaining is the case when $i=[c  n]$ and $c>\frac{\beta}{2 \alpha}$, which is dealt with below. First, we need a series of lemmas.
	\begin{lemma}
		\label{firstbound}
		Let $\tau_0(i)$ be as above, where $i=[c  n]$, $c>\frac{\beta}{2 \alpha}$. Define  $s=[\frac{\beta}{2\alpha} n]$.
		Let $C_{i\to s\to i}=\inf \{t>0;\ \eta(0)=i,\ \eta(t)=i,\ \eta(t)>\tau_i(s)  \}$ be the time it takes for the dynamic graph to go from $i$ edges to $s$ edges back to $i$ edges. Then,
		\begin{align*}
			E(\tau_0(i)) = \mathbb{E}(C_{i\to s\to i}) +  O(\log(n))
		\end{align*}
	\end{lemma}
	\begin{proof}
		By the strong Markov property we have that $\mathbb{E}(C_{i\to s\to i}) = \mathbb{E}(\tau_i(s))+\mathbb{E}(\tau_s(i))$ and $\mathbb{E}(\tau_0(i))=\mathbb{E}(\tau_0(s))+\mathbb{E}(\tau_s(i))$. By Theorem \ref{mainHittingTimes}, both $\mathbb{E}(\tau_i(s))$ and $\mathbb{E}(\tau_0(s))$ are of order $O(\log (n))$. We get,
		\begin{multline*}
			\mathbb{E}(\tau_0(i))
			= \mathbb{E}(\tau_0(s))+ \mathbb{E}(\tau_s(i)) \\
			=\mathbb{E}(\tau_0(s))+ \mathbb{E}(C_{i\to s\to i}) - \mathbb{E}(\tau_i(s)) =  \mathbb{E}(C_{i\to s\to i})+ O(\log(n)).	
		\end{multline*}
	\end{proof}
	
	Lemma \ref{firstbound} can be used to derive bounds for $\mathbb{E}(\tau_0(i))$ as $C_{i\to s\to i}$ is a cycle time for the dynamic graph, and hence results from renewal theory can be applied. Before doing so we shall need two more lemmas.
	\begin{lemma}
		\label{relative_entropy}
		Let $i = [c  n]$, $c>\frac{\beta}{2\alpha}$ and let $p_n = \frac{\beta}{\beta+(n-1)\alpha}$ be the edge probability at stationarity. Then,
		\begin{align*}
			N \cdot D\left(\frac{i}{N}||p_n\right) = n\left( c\log(\frac{2\alpha}{\beta}c)-c+\frac{\beta}{2\alpha}\right) +O(1)
		\end{align*}
		where $D(a||p) = a \log(\frac{a}{p})+(1-a) \log(\frac{1-a}{1-p})$ is the relative entropy of a Bernoulli($a$) random variable with respect to a Bernoulli($p$) random variable.
	\end{lemma}
	\begin{proof}
		We have that,
		\begin{equation}
			\label{ND1}
			N \cdot D\left(\frac{i}{N}||p_n\right) = i\log\left( \frac{i/N}{p_n}\right) +(N-i)\log \left( \frac{1-i/N}{1-p_n} \right).
		\end{equation}
		Since $ \frac{i/N}{p_n} = \frac{c}{\beta/(2\alpha)} + O(\frac{1}{n}) $ and $i = cn+O(1)$ we get that $\log \left(\frac{i/N}{p_n} \right) = \log \left(\frac{c}{\beta/(2\alpha)}\right) +O(\frac{1}{n}) $. Which implies,
		\begin{equation}
			\label{ND2}
			i\log\left( \frac{i/N}{p_n}\right) = cn\log\left( \frac{c}{\beta/(2\alpha)} \right) + O(1).
		\end{equation}
		
		We also have, again by Taylor approximation,
		\begin{align*}
			&\log\left(1-i/ N\right) = -\frac{i}{N}+O(n^{-2}) \\
			& \log(1-p_n) = -p_n+O(n^{-2}).
		\end{align*}
		Together with $Np_n = \frac{\beta}{2\alpha}n+O(1)$ this implies,
		\begin{equation}
			\label{ND3}
			(N-i)\log \left( \frac{1-i/N}{1-p_n} \right) = N p_n-i +O(1)=n\left( \frac{\beta}{2\alpha}-c\right) +O(1).
		\end{equation}
		Inserting \eqref{ND2} and \eqref{ND3} in \eqref{ND1} shows that,
		\begin{align*}
			N \cdot D\left(\frac{i}{N}||p_n\right) = n\left( c\log(\frac{2\alpha}{\beta}c)-c+\frac{\beta}{2\alpha}\right) +O(1).
		\end{align*}
	\end{proof}
	\begin{remark}
		For coming results note that $c\log(\frac{2\alpha}{\beta}c)-c+\frac{\beta}{2\alpha}>0$ if $c > \frac{\beta}{2\alpha}$.
	\end{remark}

	Finally, for proving Theorem \ref{mainHittingTimes} (c) we also need the following lemma
	\begin{lemma}
		\label{expoprobab}
		Assume $c > c' > \frac{\beta}{2 \alpha}$. Let $i = [cn]$, $j = [c'n]$ and $s = [\frac{\beta}{2 \alpha}]$.
		Then there exists $\gamma = \gamma(c',c)>0$ such that $\mathbb{P}(\tau_j(i) <\tau_j(s)) < e^{-\gamma n}$ for all large enough $n$. That is, the probability that after starting in $j$, $\{\eta(t), t\geq 0\}$ reaches $i$ before $s$ is exponentially small in $n$.  
		
		Furthermore, $\mathbb{P}(\tau_{i-1}(s) < \tau_{i-1}(i))>\bar{\delta} >0$. 
	\end{lemma}
	\begin{proof}
		If $\eta(t) \geq \frac{\beta}{2 \alpha} n$, the probability that the next change in  $\eta(t)$ is upwards is at most $1/2$, while if $\eta(t) \geq j$, then the probability that the next jump is upward is at most $p(j,j+1)=\frac{\beta (N-j)}{\beta(N-j)+(n-1)j \alpha}$,  
		which converges to $ \frac{\beta}{\beta + 2c'\alpha} = \frac{1}{2(1+\delta)}$ for $\delta = \frac{1}{2}(c'\frac{2 \alpha}{\beta} -1)$, which is strictly positive since $ c' > \frac{\beta}{2 \alpha}$. 
		This implies that for $k\geq j$ and $\delta' \in (0,\delta)$, we have $p(k,k+1) \leq \frac{1}{2(1+\delta')}.$ 
		
		For $s <j <i$,  let $\tau_j(i,s) = \min(\tau_j(i),\tau_j(s))$ be the first time $\eta(t)$ hits the boundary of the interval $[s,i]$. Since we are only interested in $\mathbb{P}(\eta(\tau_j(i,s)) = s)$, we can consider the process $\{\eta'(t),t\geq 0\}$, where $\eta'(t) = \eta(t)$ if $t\leq \tau_j(i,s)$  and $ \eta'(t) = \eta(\tau_j(i,s))$ otherwise. We note that $\{\eta'(t),0 \leq  t \leq \tau_j(i,s)\} = \{\eta(t),0 \leq  t \leq \tau_j(i,s)\}$. 
		
		Let $\delta' \in (0,\delta)$ and $n$ be large enough, so that $p(j,j+1)<\frac{1}{2(1+\delta')}$.
		We now check that the process defined through $\eta'(0) =j$ and 
		$$y(t)= (\eta'(t)- j) \ind(\eta'(t) \leq j) + (1+\delta')^{\eta'(t)-j} \ind(\eta'(t) > j)$$ 
		is a supermartingale.
		Indeed, if $\eta'(t)<j$, then the probability that the next jump is up is less than $1/2$, so trivially, the expected $y$ value after the next jump is bounded above by $\frac{1}{2}(y(t)+1) + \frac{1}{2} (y(t)-1) = y(t)$.
		If $\eta'(t)\geq j$ the probability that the next jump is up is less than $\frac{1}{2(1+\delta')}$. So if $\eta'(t)=j$, the expected $y$ value after the next jump is bounded above by
		$$(1+\delta') \frac{1}{2(1+\delta')} - \left(1- \frac{1}{2(1+\delta')}\right) = \frac{1}{2(1+\delta')} - \frac{1}{2}<0 = y(t)|_{\eta'(t)=j}.$$
		While for $\eta'(t) > j$ this expectation is bounded above by
		\begin{multline*}
			y(t) (1+\delta') \frac{1}{2(1+\delta')} + y(t) \frac{1}{1+\delta'} 
			\left(1- \frac{1}{2(1+\delta')}\right) \\
			= \frac{y(t)}{2(1+\delta')^2}\left((1+\delta')^2 + (2(1+\delta')- 1)\right)
			= \frac{y(t)}{2(1+\delta')^2}(2 (1+\delta')^2 -\delta'^2)
			< y(t).
		\end{multline*}
		Let $\mathcal{F}'_t$ be the $\sigma$-algebra generated through $\{\eta'(s),s \in [0,t]\}$. The above computations implies that $\mathbb{E}(y(s)|\mathcal{F}'_t) \leq y(t)$, for all $s>t$ and therefore $\{y(t),t\geq 0\}$ is a supermartingale.

		We obtain by optional stopping arguments that
		\begin{multline*}
			0 = y(0)\geq \mathbb{E}(y(\tau_j(i,s))) \\
			=  \mathbb{P}(y(\tau_j(i,s))=i)  (1+\delta')^{i-j} - \left(1- \mathbb{P}(y(\tau_j(i,s))=i)\right)  (j-s)\\ 
			= \mathbb{P}(y(\tau_j(i,s))=i)  \left((1+\delta')^{i-j} +(j-s)\right)- (j-s),
		\end{multline*}
		which implies 
		$$\mathbb{P}(\tau_j(i) < \tau_j(s))= \mathbb{P}(\eta(\tau_j(i,s)) =i) \leq  \frac{j-s}{(1+\delta')^{i-j} +j-s} \leq \frac{(c'-\frac{\beta}{2 \alpha}) n+1}{(1+\delta')^{(c-c')n-1}}.$$ 
		Setting $\gamma \in (0,(c-c')\log(1+\delta'))$ completes the proof of the first part of the lemma. 
		
		In order to prove the second part of the lemma note that if $\eta(0) = \eta'(0) = i-1$, then $y(0) = (1+\delta')^{i-1-j}$. Using the same supermartingale argument as above, we obtain that
		$$
		(1+\delta')^{i-1-j} 
		\geq 
		\left(1-\mathbb{P}(\tau_{i-1}(s) < \tau_{i-1}(i))\right) (1+\delta')^{i-j} - \mathbb{P}(\tau_{i-1}(s) < \tau_{i-1}(i)) (j-s),
		$$
		which implies $\mathbb{P}(\eta(\tau_{i-1}(s)) < \tau_{i-1}(i))) \geq \frac{\delta'(1+\delta')^{i-1-j}}{(1+\delta')^{i-j}+ (j-s)}$. Because $c > c' > \frac{\beta}{2 \alpha}$ and $i = [cn]$, $j = [c'n]$ and $s = [\frac{\beta}{2 \alpha}]$, the latter expression converges to $\frac{\delta'}{1+\delta'}>0.$ Choosing $\bar{\delta} \in (0, \frac{\delta'}{1+\delta'})$ completes the proof of the lemma.
	\end{proof}

	We can now prove Theorem \ref{mainHittingTimes} (c).
	\begin{proof}[Proof of Theorem \ref{mainHittingTimes} (c)]
		Observe that $\{\eta(t),\ t\geq 0 \} $ is a regenerative process. 
		Let $T_{\geq i} = \int_{0}^{C_{i\to s\to i}} 1\{\eta(t)\geq i \} dt $ be the time $\{\eta(t),\ t\geq 0 \} $ spends above state $i$ in a $(i\to s\to i)$ cycle, where $s = [\frac{\beta}{2\alpha}   n]$. 
		Since $C_{i\to s\to i}$ is a renewal time for $\{\eta(t),\ t\geq 0 \} $ we have, by basic renewal theory,
		\begin{align*}
			\lim\limits_{t \to \infty} \mathbb{P}(\eta(t)\geq i) = \frac{\mathbb{E}( T_{\geq i})}{\mathbb{E}(C_{i\to s\to i})}.
		\end{align*}
		If $T_s$ is a strong stationary time for the dynamic graph, then $\lim\limits_{t \to \infty} \mathbb{P}(\eta(t)\geq i) =  \mathbb{P}(\eta(T_s)\geq i)$. Hence,
		\begin{align}
			\label{regeq}
			\mathbb{E}(C_{i\to s\to i})	= \frac{\mathbb{E}( T_{\geq i} )}{ \mathbb{P}(\eta(T_s)\geq i) }.
		\end{align}
		\begin{figure}[H]
			\includegraphics{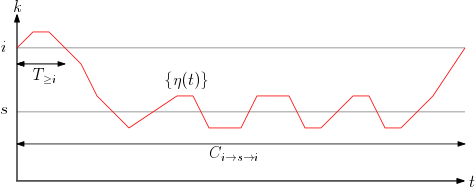}
			\caption{Visual Aid for Equation \eqref{regeq}}
		\end{figure}
		Since $\eta(T_s)\sim$ Bin($N, p_n)$, with $p_n = \frac{\beta}{\beta + (n-1)\alpha}$ we can give upper and lower bounds for $\mathbb{P}(\eta(T_s)\geq i)$, see \cite[p. 114]{ash}.
		\begin{align*}
			( 8 i  )^{-1/2} \exp\{-N \cdot D\left( \frac{i}{N}||p_n\right) \} \leq \mathbb{P}(\eta(T_s)\geq i) \leq \exp\{-N \cdot D\left( \frac{i}{N}||p_n\right) \}
		\end{align*}
		where $D(a||p) = a \log(\frac{a}{p})+(1-a) \log(\frac{1-a}{1-p})$. 
		
		Next we show that $\mathbb{E}(T_{\geq i}) = \Theta(n^{-1})$. Let $H_i$ be the holding time in state $i$. Then,
		\begin{align*}
			E(T_{\geq i})\geq E(H_i) = \frac{ n-1 }{ (N-i)\beta+i(n-1)\alpha  } > K n^{-1},  
		\end{align*}
		for some $K >0$ and $n$ large enough. By the strong Markov property of the process $\{\eta(t),\ t\geq 0 \} $ we have that,
		\begin{align*}
			\mathbb{E}(T_{\geq i}) = p(i,i+1) ( \mathbb{E}(H_i)+\mathbb{E}(\tau_{i+1}(i))+E(T_{\geq i}))+p(i,i-1) ( \mathbb{E}(H_i) + q_{i-1}  \mathbb{E}(T_{\geq i})	)
		\end{align*}
		where $q_{i-1}=\mathbb{P}(\tau_{i-1}(i)<\tau_{i-1}(s))$ is the probability of reaching state $i$ before state $s$, if $\{\eta(t),\ t\geq 0 \} $ starts in state $i-1$. We get,
		\begin{align*}
			\mathbb{E}(T_{\geq i})  = \frac{  \mathbb{E}(H_i) +p(i,i+1)\mathbb{E}(\tau_{i+1}(i)) }{  p(i,i-1) (1-q_{i-1})   }
		\end{align*}
		Now, $\mathbb{E}(H_{i+k}) = \frac{ n-1 }{ (N-i-k)\beta+(i+k)(n-1)\alpha  } = \Theta(n^{-1})$, for fixed $k$ and $i=[c  n]$.
		Furthermore, 
		\begin{align*}
			&p(i+k,i+k+1) \to \frac{\beta}{\beta + 2c\alpha} \text{ as } n \to \infty\\
			&p(i+k,i+k-1) \to \frac{ 2c\alpha}{\beta + 2c\alpha} \text{ as } n \to \infty
		\end{align*}
		again for fixed $k$ and $i=[c  n]$.
		
		Secondly, since $\mathbb{E}(\tau_{i+2}(i+1))	< \mathbb{E}(\tau_{i+1}(i))	$ we get,
		\begin{align*}
			\mathbb{E}(\tau_{i+1}(i)) &= \mathbb{E}(H_{i+1}) + p(i+1,i+2) (\mathbb{E}(\tau_{i+2}(i+1))	+ \mathbb{E}(\tau_{i+1}(i))	)\\
			& \leq E(H_{i+1}) + 2p(i+1,i+2)   \mathbb{E}(\tau_{i+1}(i)) 
		\end{align*}
		Since $2p(i+1,i+2)<1$ we get,
		\begin{align*}
			\mathbb{E}(\tau_{i+1}(i)) \leq \frac{ \mathbb{E}(H_{i+1})  }{1-2p(i+1,i+2)} = O(n^{-1}).
		\end{align*}
		In the second part of Lemma \ref{expoprobab}  we show that $1-q_{i-1} \not \to 0$ as $n \to \infty$. 
		Therefore, $\mathbb{E}(T_{\geq i}) = O(n^{-1})$.
		
		Combining all of the above we get,
		\begin{align*}
			\Theta(n^{-1}) \exp\left\lbrace N \cdot D\left(\frac{i}{N}||p\right)\right\rbrace  \leq \mathbb{E}(C_{i\to s\to i}) \leq \Theta(n^{-1/2}) \exp\left\lbrace N \cdot D\left(\frac{i}{N}||p\right)\right\rbrace. 
		\end{align*}
		By Lemma \ref{firstbound} we have 
		\begin{align*}
			\mathbb{E}(C_{i\to s\to i})-O(\log (n)) \leq \mathbb{E}(\tau_0(i)) \leq \mathbb{E}(C_{i\to s\to i})+O(\log(n)),
		\end{align*}
		which together with Lemma \ref{relative_entropy} implies, 
		\begin{multline*}
			\Theta(n^{-1}) \exp\left\{n\left( c\log(\frac{2\alpha}{\beta}c)-c+\frac{\beta}{2\alpha}\right)\right\rbrace  \leq \mathbb{E}(\tau_0(i))  \\
			\leq	\Theta(n^{-1/2}) \exp\left\{n\left( c\log(\frac{2\alpha}{\beta}c)-c+\frac{\beta}{2\alpha}\right) \right\rbrace .
		\end{multline*}

		What remains to prove is that for $n\to \infty$,  $\tau_0([c n])/\mathbb{E}[\tau_0([c n])]$ converges in distribution to an exponential random variable with mean 1.
		
		To do this we need some definitions and assumptions.
		Let $T_0=0$ and $T_{i+1}$ be the smallest time after $T_i$, such that in
		the interval $I_{i+1}=(T_i,T_{i+1}]$ all edges are updated in the sense of (iia) of the defining properties of the dynamic Erd\H{o}s-R\'enyi graph. The lengths of those intervals are trivially i.i.d.
		Furthermore, let $A_i$ be the event
		$\{\max_{t \in I_i} \eta(t) > [c n]\}$.
		Note that $A_i$ may be dependent on $A_{i-1}$ and $A_{i+1}$ but is independent
		of  $A_j$ for $|j-i| \geq 2$ and independent of $T_j$ for all $j \geq 1$. 
		Furthermore, since $\eta(T_i)$ is binomially distributed with parameters $N$ and $p= \frac{\beta}{\beta + \alpha(n-1)}$ for all  $i \geq 2$, $\mathbb{P}(A_i)$ is the same for all $i \geq 2$ and we denote this probability by $q = q(n)$. 
		
		Let $J = \min\{i: A_i \text{ occurs}\}$ and $J' = \min\{j: A_{2j} \text{ occurs}\}$. Note that $J \leq 2J'$ and $J'$ is geometrically distributed with parameter $q$. In particular,  $2/q =  2\mathbb{E}(J') \geq \mathbb{E}(J)$. Furthermore, $$\mathbb{E}(\tau_0([cn]) = \mathbb{E}(T_1)\mathbb{E}(J) + O(\log (n)) = O(\log (n)) \mathbb{E}(J)$$ by the independence of the $A_i$'s and $J$ and by $\mathbb{E}(T_1) = O(\log (n))$.
		Together this  implies that $q$ is exponentially small in $n$, since $\mathbb{E}(\tau_0([cn])$ grows exponentially in $n$.
		
		Our strategy of proof is to show that, 
		\begin{equation}
			\label{qineq1}
			\mathbb{P}(A_i|A_j^c, j=1,2,\dots,i-1) \leq q \qquad \text{for  $i \geq 1$}
		\end{equation}
		and  
		\begin{equation}
			\label{qineq2}
			\mathbb{P}(A_i|A_j^c, j=1,2,\dots,i-1) \geq q(1+o_n(1)) \qquad \text{for  $i \geq 3$.}
		\end{equation}
		Once this is proved, we obtain immediately that $J$ dominates a geometric random variable with parameter $q$ (say $J^-$ and $J$ conditioned on, is dominated by  a geometric random variable with parameter $q(1+o_n(1))$ plus 2, say $J^+$. 
		Note that $\mathbb{E}(J^-) = \frac{1}{q}$ 
		and 
		$\mathbb{E}(J^+) =  2 +  \frac{1}{q(1+o_n(1)} = \frac{1}{q}(1+o_n(1))$.
		So, it follows that 
		$$\frac{1}{q} \leq \mathbb{E}(J) \leq  \frac{1}{q}(1+o_n(1)).$$ 
		
		Since for all $i \geq 1$, we have $T_{i+1}-T_i = O(\log(n))$ w.h.p.\ and  $\mathbb{E}(T_{i+1}-T_i) = O(\log(n))$ and the events $A_i$ are independent of the times $T_i$ and by $T_J \to \infty$ w.h.p., we have (also w.h.p.)
		$$\frac{\tau_0([c n])}{\mathbb{E}(\tau_0([c n]))} =
		\frac{T_J + O(\log(n))}{\mathbb{E}[T_J]+O(\log(n))} =
		\frac{J \cdot T_J/J +O(\log(n))}{\mathbb{E}(T_1)\mathbb{E}(J) +O(\log(n))} = \frac{J}{\mathbb{E}(J)}\frac{T_J}{J\mathbb{E}(T_1)} (1+ o_n(1)).$$
		
		By the strong law of large numbers and $J \to \infty$ w.h.p.,  $\frac{T_J}{J\mathbb{E}(T_1)} = 1+o_n(1)$ w.h.p. So, $$\frac{\tau_0([c n])}{\mathbb{E}(\tau_0([c n])} = \frac{J}{\mathbb{E}(J)}(1+ o_n(1)) \qquad \text{w.h.p.}$$
		
		By standard results on geometric distributed random variables $J^-/\mathbb{E}(J^-)$ converges in distribution to an exponential random variable with parameter 1.
		Similarly we deduce that  $(J^+-2)/\mathbb{E}(J^+-2)$ converges in distribution to an exponential random variable with parameter 1. Since $J \to \infty$ w.h.p., this also implies that $(J^+)/\mathbb{E}(J^+)$ converges in distribution to an exponential random variable with parameter 1. Recalling  $J$ is stochastically sandwiched between $J^-$ and $J^+$ and 
		$$\mathbb{E}(J) = \mathbb{E}(J^-)(1+o_n(1)) =\mathbb{E}(J^+)(1+o_n(1)),$$ we obtain that $\frac{J}{\mathbb{E}(J)}$ converges in distribution to  an exponential random variable with parameter 1.
		Therefore, $\frac{\tau_0([c n])}{\mathbb{E}(\tau_0([c n]))}$ converges in distribution to an exponential random variable with parameter 1. 
		
		The only things left to prove are inequalities (\ref{qineq1}) and (\ref{qineq2}).  
		To prove (\ref{qineq1}), first note that  for $i \geq 2$, $$\mathbb{P}(A_i|A_j^c, j=1,2,\dots,i-1) \leq \mathbb{P}(A_i) = q,$$ since for all $t>0$, the process $\{\eta(t+u),u\geq 0\}$ stochastically dominates the process $\{\eta(t+u),u\geq 0|\eta(t)\leq [cn]\}$, where the latter is equal to the former conditioned on  $\eta(t)\leq [cn]$. Since $\eta(0) =0$ it also follows that $\mathbb{P}(A_1) \leq q$. Which completes the prove of inequality (\ref{qineq1}).
		
		Using the same line of argument for inequality (\ref{qineq2}) we obtain for $i \geq 3$
		\begin{multline*}
			\mathbb{P}(A_i|A_j^c, j=1,2,\dots,i-1) = \frac{\mathbb{P}(A_i,A_{i-1}^c|A_j^c, j=1,2,\dots,i-2)}{\mathbb{P}(A_i|A_j^c, j=1,2,\dots,i-2)} \\
			\geq \mathbb{P}(A_i,A_{i-1}^c|A_j^c, j=1,2,\dots,i-2)\\ 
			= \mathbb{P}(A_i|A_j^c, j=1,2,\dots,i-2) - \mathbb{P}(A_i,A_{i-1}|A_j^c, j=1,2,\dots,i-2)\\ \geq 
			\mathbb{P}(A_i|A_j^c, j=1,2,\dots,i-2) - \mathbb{P}(A_i,A_{i-1})\\= \mathbb{P}(A_i|A_j^c, j=1,2,\dots,i-2)-q
			\mathbb{P}(A_i|A_{i-1}).
		\end{multline*}
		Because $A_i$ is independent of $A_j$ for $|j-i|\geq 2$, we have  $\mathbb{P}(A_i|A_j^c, j=1,2,\dots,i-2)=q$ and therefore, $\mathbb{P}(A_i|A_j^c, j=1,2,\dots,i-1) \geq q(1-\mathbb{P}(A_i|A_{i-1}))$.
		So, our proof is complete if we show that $\mathbb{P}(A_i|A_{i-1}) = \mathbb{P}(A_3|A_{2})= o_n(1)$.

		Define the following events: $B_1$ is the event that $T_2-T_1 > \sqrt{\log(n)}$ and $B_2$ is the event that the first hitting time of $[cn]$ is before $T_2 - 2\frac{\log 2}{\alpha}$.
		We show that the probabilities $\mathbb{P}(A_3| B_1, B_2,A_2)$, $\mathbb{P}(B_1^c|A_2)$ and   $\mathbb{P}(B_2^c|A_2,B_1)$ are all $o_n(1)$. Combining this with 
		$$\mathbb{P}(A_3|A_2) \leq \mathbb{P}(A_3| B_1, B_2,A_2) + \mathbb{P}(B_1^c|A_2) + \mathbb{P}(B_2^c|A_2,B_1)$$
		gives the desired result.
		
		From $\frac{\alpha (T_2-T_1)}{2 \log(n)} \xrightarrow[]{p} 1$, it immediately follows that  
		$\mathbb{P}(B_1^c|A_2) = o_n(1)$.
		
		Since for $i\geq 2$,  $\eta(T_{i-1})$ has the stationary distribution, conditioned on $A_i$, the first time in $I_i$, $\eta(t) \geq [cn]$ is stochastically dominated by a uniform ramdom variable on the times when edges are updated in $I_i$.
		This gives that  $\mathbb{P}(B_2^c|A_2,B_1) = o_n(1)$.
		
		Let $s= [\frac{\beta}{2 \alpha} n]$ and $j = \left[\frac{[cn] + s}{2}\right]$. For $c>\frac{\beta}{2 \alpha}$,
		if $\eta(t_0)=[cn]$, then by Lemma \ref{stationaryedgeslow},  
		$t_1 = \inf\{t\geq t_0; \eta(t)= k\} \xrightarrow[]{p}  t_0 + \frac{\log(2)}{\alpha}$ as $n \to \infty$.
		By Lemma \ref{expoprobab} we know that  $$\mathbb{P}(\text{$\{\eta(t), t \geq t_1\}$ reaches $[cn]$ before it reaches $s$})$$  is exponentially small in $n$.
		We already know that $\tau_s([cn])$ is w.h.p. exponentially large in $n$.
		This gives that $\mathbb{P}(A_3| B_1, B_2,A_2) = o_n(1)$ and the proof is complete.
	\end{proof}
	
	Because of the similarities between the dynamic Erd\H{o}s-R\'enyi graph and a SIS epidemic, we mention here that the methods used in \cite{Andersson98} might be used to derive similar results for hitting times.
	\subsection{The size of the largest component.}
	\label{sizecomponent}
	When studying any type of random graph the size of the largest component is often of interest. For instance, we might ask how long it will take for the size of the largest component in the dynamic graph to \textit{exceed}, say, $\epsilon  n$? We suggest that one way to approach this problem is through the edge process $\{\eta(t),\ t\geq 0 \}$. 
	
	In a static Erd\H{o}s-R\'enyi graph there is an intimate connection between the number of edges in the graph and the size of the largest component. Namely, if $G(n,M(n))$ is a static Erd\H{o}s-R\'enyi graph with $n$ vertices and $M(n)$ edges, where $M(n) = [c n]$, then as $n \to \infty$, the size of the largest component exhibits three different behaviors depending on $c$. The following holds with high probability: (i) if $c<1/2$, called the subcritical case, then the size of the largest component if of order $\log (n)$; (ii) if $c>1/2$, called the supercritical case, it is of order $n$; (iii) if $c=1/2$, called the critical case, then the largest component is of order $n^{2/3}$, see \cite[p.130]{Boll}. 
	A modification of a classical result from Erd\H{o}s-R\'enyi \cite{ER} gives the following lemma for the supercritical case:
	\begin{corollary}
		\label{sizeCorollary}
		Let $\{ \mathcal{G}(n,M(n)) \}$ be a sequence of Erd\H{o}s-R\'enyi graphs with $n$ vertices and $M(n)$ edges. Let $|C(n,M(n))|$ be the size of the largest component of $\mathcal{G}(n,M(n))$. \\
		Then, for every $0 < \epsilon < 1$ there exist a $c_{\epsilon}>\frac{1}{2}$ such that if $M(n) = [c_{\epsilon} n]$ then, 
		\[
		\frac{|C(n,M(n))|}{n} \xrightarrow[]{p} \epsilon \text{ as } n \to \infty.
		\]
		Furthermore,
		\[
		c_{\epsilon} = \frac{- \log (1-\epsilon)}{2  \epsilon}.
		\]
	\end{corollary}
	Hence, if $c = \frac{- \log (1-\epsilon)}{2  \epsilon}$, then $|C(n, [c n])| \approx n \epsilon$ with high probability. However, 
	$
	\frac{|C(n,[c n])|}{n} \xrightarrow[]{p} \epsilon \text{ does not imply } P(|C(n,[c n])|\geq n\epsilon)\to 1.
	$
	However, for the latter to hold we may just choose $M(n)=[c'n]$, for any $c'>c$.
	
	This in combination with Theorem \ref{mainHittingTimes} can be used  to provide a (asymptotic) bound on the expected time it takes for the size of the largest component in the dynamic Erd\H{o}s-R\'enyi graph to \textit{exceed} $\epsilon  n$, for given $\epsilon > 0$. 
	In a static Erd\H{o}s-R\'enyi we would need more than $[c n]$ edges, say $[c'n]$ edges, where $c'>c = \frac{- \log (1-\epsilon)}{2  \epsilon}$ for this to be very likely---e.g. we can take $c'=\frac{- \log (1-(\epsilon+\delta))}{2  (\epsilon+\delta)}$, $\delta\in(0,1-\epsilon)$, so that the fraction of vertices in the largest component converges in probability to $\epsilon+\delta$. For the dynamic graph, we can instead wait until that many edges has appeared, and be very certain that the size of the largest component has exceeded $\epsilon  n$ no later than that time. This leads to the following lemma:
	\begin{lemma}
		\label{connectionLemma}
		Let $\hat{\tau}(\epsilon n)$ be the first time  the dynamic Erd\H{o}s-R\'enyi graph, starting with no edges, has a component of size at least $\epsilon n$, $\epsilon\in(0,1)$. Let $\tau_0(i)$ be the time it takes for the dynamic graph to reach, starting with no edges, $i$ edges.
		Then for all $c'>c = \frac{-\log(1-\epsilon)}{2\epsilon}$,
		\begin{align*}
			\mathbb{E}[\hat{\tau}(\epsilon  n)] = O(\mathbb{E}[\tau_0([c' n])]).
		\end{align*}
	\end{lemma}
	\begin{proof}
		Let $A$ be the event that at time $\tau_0([c' n])$ the largest component of the graph is at least size $\epsilon  n$ and let $A^c$ be the complement of $A$. Observe that $A$ is independent of $\tau_0([c' n])$, since the edge processes are all independent and therefore all edge configurations have the same probability at time $\tau_0([c' n])$. By Corollary \ref{sizeCorollary} $P(A) \to 1$ as $n \to \infty$. We note that,
		\begin{multline*}
			\mathbb{E}[\hat{\tau}(\epsilon  n)] =  
			\mathbb{E}[\hat{\tau}(\epsilon  n)|A]\mathbb{P}(A) + 
			\mathbb{E}[\hat{\tau}(\epsilon  n)|A^c]\mathbb{P}(A^c)\\
			\leq  \mathbb{E}[\tau_0([c'  n])]\mathbb{P}(A) + 
			(\mathbb{E}[\hat{\tau}(\epsilon  n)] + \mathbb{E}[\tau_0([c'  n])])\mathbb{P}(A^c)\\
			= \mathbb{E}[\tau_0([c'  n])] + \mathbb{E}[\hat{\tau}(\epsilon  n)]\mathbb{P}(A^c).
		\end{multline*}
		Hence,
		\[
		\mathbb{E}[\hat{\tau}(\epsilon  n)] \leq \frac{\mathbb{E}[\tau_0(c'n)]}{1-\mathbb{P}(A^c)}.
		\]
		It follows that $\mathbb{E}[\hat{\tau}(\epsilon  n)] = O(\mathbb{E}[\tau_0([c' n])])$.
	\end{proof}
	
	We can now prove Corollary \ref{compsizecor}.
	\begin{proof}[Proof of Corollary \ref{compsizecor}]
		In this proof order terms are for $\hat{\epsilon} \to 0$.\\
		Recall that we consider the critical case $\alpha=\beta$. Let $c' = \frac{-\log(1-\hat{\epsilon})}{2\hat{\epsilon}}$, $\hat{\epsilon}>\epsilon$, so that $c'>c$. Note,
		by Taylor approximation, $$c' = \frac{-\log(1-\hat{\epsilon})}{2\hat{\epsilon}} = 1/2+\hat{\epsilon}/4+\hat{\epsilon}^2/6+O(\hat{\epsilon}^3)>1/2 = \frac{\beta}{2\alpha}.$$ 
		Hence, by Lemma \ref{connectionLemma} and Theorem \ref{mainHittingTimes} (c),
		\begin{align*}
			\mathbb{E}[\hat{\tau}(\epsilon  n)] = O(n^{-1/2})e^{n(c'\log(2c') +  1/2-c'   )}.
		\end{align*}
		Recall, $\log(1+x) = \sum_{k=1}^{\infty} (-1)^{n+1}\frac{x^n}{n}$ if $|x|<1$. Hence,
		\begin{multline*}
			\log(2c') = \log(1 +(2c'-1)) = (2c'-1)+(2c'-1)^2/2+O(c'^3)
			= \hat{\epsilon}/2 + 5\hat{\epsilon}^2/24+O(\hat{\epsilon}^3).
		\end{multline*}
		if $|2c'-1|<1$ (which can be solved numerically and implies that $ \epsilon < 0.7968$).
		Hence,
		\begin{align*}
			c'\log(2c')-c'+1/2 = \hat{\epsilon}^2/16+O(\hat{\epsilon}^3)
		\end{align*}
		and the corollary follows.
	\end{proof}
	
	We now show another way of deriving the expected time until a large component appears, in the critical case setting---which is indeed sharper. 
	Theorem 3.1 of \cite{OCon98} gives that for the static Erd\H{o}s-R\'enyi graph $G(n,p)$ with $p=1/n$, 
	and $C(n,p)$ the size of its largest connected component, $$\frac{1}{n}\log\left(\mathbb{P}\left(\frac{C(n,p)}{n} \geq \epsilon\right)\right)  \to -I_1(\epsilon) = -\epsilon^3/8 + O(\epsilon^4),$$
	where $I_1(x) = - x \log(1-e^{-x}) + x \log(x) + (1-x) \log(1-x) +x(1-x) $ is the rate function of the appropriate large deviation principle as given in \cite{OCon98}.
	We know that in time $O(\log(n))$ all edges are updated, in the sense of defining property (iia) of the dynamic Erd\H{o}s-R\'enyi graph and the strong stationary time in Proposition \ref{Fastesttimetostationary}. Let $T_0=0$ and for $i =1,2,\cdots$, let $T_i$ be the first time after $T_{i-1}$ that updates have occurred in the interval $(T_{i-1},T_i]$ at all edges. Note that for $i=1,2,\cdots$ the graphs $G(T_i)$ are independent of each other and of the $T_i$s by the definition of the $T_i$s. Furthermore, each  $G(T_i)$ contains a component of size at least $\epsilon n$ with probability $\mathbb{P}(n^{-1}C(n,p) \geq \epsilon)$, and therefore $$\mathbb{E}[\hat{\tau}(\epsilon  n)] \leq O(\log(n))/\mathbb{P}(n^{-1}C(n,p) \geq \epsilon) = e^{n (\epsilon^3/8+  O_{\epsilon}(\epsilon^4)) + o_n(n) }$$ where $f(n) = o_n(n)$ means that $\lim\limits_{n\to \infty} f(n)/n = 0$.\\
	It is also easy to show that $I(\epsilon) < c_{\epsilon}  \log (2c_{\epsilon})+1/2-c_{\epsilon}$, where $c_{\epsilon} = \frac{-\log (1-\epsilon)}{2\epsilon}$, see figure 2, showing that the O'Connell approach produces a sharper bound. From comparing the bounds we deduce that, in the critical case of the dynamic graph, a large component emerges through an unlikely configuration a few edges. Where we with few edges mean less than expected in the static setting. 
	\begin{figure}[H]
		\leftskip -1cm
		\includegraphics[width = 15cm]{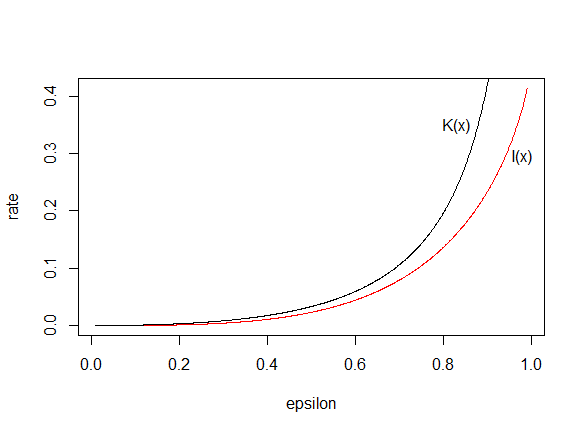}
		\caption{$K(\epsilon) = c_{\epsilon}  \log (2c_{\epsilon})+1/2-c_{\epsilon}$ (black) and $I(\epsilon)= - \epsilon \log(1-e^{-\epsilon}) + \epsilon \log(\epsilon) + (1-\epsilon) \log(1-\epsilon) +\epsilon(1-\epsilon)$ (red) }
	\end{figure}

	\section*{Acknowledgments}
	We would like to thank Mia Deijfen for helpful comments on the manuscript. 
	\bibliography{articleRef}{}
	\bibliographystyle{siam}

\end{document}